\documentclass{article}
\usepackage{latexsym, amsmath, amssymb}
\usepackage{graphicx}
\usepackage{epstopdf}
\usepackage{amsmath, amsthm, amscd, amsfonts}
\usepackage{amsmath}
\usepackage{amssymb}
\usepackage{array}
\usepackage{multirow}
\usepackage{float}
\usepackage{enumerate}
\usepackage[a4paper, total={6in, 10in}]{geometry}

\newtheorem{thm}{Theorem}[section]

\newtheorem{cor}[thm]{Corollary}

\theoremstyle{definition} 
\theoremstyle{question} 
\theoremstyle{remark} 

\makeatletter
\newcommand*{\rom}[1]{\expandafter\@slowromancap\romannumeral #1@}
\makeatother

\def\bege{\begin{equation}} \def\ende{\end{equation}}

   \def\begr{\begin{eqnarray}}
\def\endr{\end{eqnarray}} 
 
\def\bege{\begin{equation}} \def\ende{\end{equation}}
\def\begr{\begin{eqnarray}} \def\endr{\end{eqnarray}} \def\bnum{\begin{enumerate}} \def\enum{\end{enumerate}}
\begin{document}
\begin{center}
\textbf{On Fault-Tolerant Resolvability of Double Antiprism and its related Graphs}
\end{center}

\begin{center}
Sunny Kumar Sharma$^{1}$ and Vijay Kumar Bhat$^{2,}$$^{\ast}$
\end{center}

\begin{center}
School of Mathematics, Faculty of Sciences, Shri Mata Vaishno
Devi University,\end{center}\begin{center}Katra-$182320$, Jammu and
Kashmir, India.
\end{center}

\begin{center}
  $^{1}$sunnysrrm94@gmail.com, $^{2}$vijaykumarbhat2000@yahoo.com
\end{center}
\textbf{Abstract}
For a connected graph $\Gamma=(V,E)$, a subset $R$ of ordered vertices in $V$ is said to be a resolving set in $\Gamma$, if the vector of distances to the vertices in $R$ is unique for each $u^{i}\in V(\Gamma)$. The metric dimension of $\Gamma$ is the minimum cardinality of such a set $R$. If $R\setminus \{u^{i}\}$ is still a resolving set $\forall$ $u^{i}\in R$, then $R$ is called a fault-tolerant resolving set (FTRS) for $\Gamma$ and its least cardinality is the fault-tolerant metric dimension (FTMD) of $\Gamma$. In this article, we introduce the concept of an independent fault-tolerant resolving set (IFTRS) and investigate it for several well-known graphs. We also show that the FTMD is four for three closely related families of convex polytopes available in the literature (viz., double antiprism $\mathbb{A}_{n}$, $S_{n}$, and $T_{n}$).  
$^{\ast}$ Corresponding author.\\\\
\textbf{{\Large $1$. Introduction and Preliminaries}}\\\\
Let $\Gamma=(V,E)$ be a non-trivial, undirected, simple and connected graph with the edge set $E(\Gamma)$ and the vertex set $V(\Gamma)$. The distance between two vertices $u^{1}, u^{2} \in V$, denoted by $d(u^{1}, u^{2})$, is the length of a shortest $u^{1}-u^{2}$ path in $\Gamma$. The number of edges that are incident to a vertex $u$ in $\Gamma$, denoted by $d_{u}$, is the degree of that vertex. For $n\geq3$, $P_{n}$, $C_{n}$, and $K_{n}$ denote the path graph, the cycle graph, and the complete graph respectively. Let $R=\{u^{1}, u^{2}, u^{3},...,u^{t}\}$ be an ordered set of vertices in $\Gamma$ and $u\in V$. Then, the co-ordinate (or code) ${\ss}(u|R)$ of $u$ with respect to $R$ is the $t$-tuple $(d(u, u^{1}), d(u, u^{2}), d(u, u^{3}),...,d(u, u^{t}))$. The subset $R$ of $\Gamma$ is said to be a resolving set (or metric generator) for $\Gamma$, if ${\ss}(u^{2}|R) \neq {\ss}(u^{1}|R)$, for any pair of vertices $u^{1} \neq u^{2} \in V(\Gamma)$. A metric basis (MB) for $\Gamma$ is a resolving set (RS) with a minimum number of elements, and the cardinality of the MB is called the metric dimension (MD) of $\Gamma$, denoted by $dim(\Gamma)$.\vspace{0.7mm}

The concept of MD was introduced independently by Slater \cite{ps} and Harary \& Melter \cite{fr}. Since then, the problem of MD has received a lot of attention. The notions of RS and MD have proven useful in various areas such as robotic navigation, the structure of chemical compounds, combinatorial optimization, image processing \& pattern recognization, connected joins in networks, pharmaceutical chemistry, game theory, etc. for these see \cite{ns, srb, mt, sv}.\vspace{0.7mm}

Hernando et al. \cite{hck} computed the MD of Fan graph $F_{n}$ ($n\geq7)$, and proved that $dim(F_{n})=\lfloor \frac{2n+2}{5}\rfloor$. In \cite{bck}, Buczkowski et al. proved that for Wheel graph $W_{n}$ ($n\geq7)$, the MD is $\lfloor \frac{2n+2}{5}\rfloor$. Tomescu and Javaid \cite{tck} proved that $dim(J_{2n})=\lfloor \frac{2n}{3}\rfloor$, for the Jahangir graph $J_{2n}$ ($n\geq4$). After that, the MD for several well-known graphs have been investigated such as prism graph, convex polytopes, antiprism graph, unicyclic graph, Petersen graph, flower graph, some regular graphs, etc. see \cite{itm, co, sv} and references therein.\vspace{0.7mm}

Recent advancements in the definition of MD have paved the path for a new related concept called fault tolerance. Assume that a network has $n$-processing elements (or units). To create a self-stable fault-tolerant system, it is necessary that if any single unit fails (or crashes), another chain of units (with the exception of the faulty element) will replace the initial (or original) chain. As a result, the nature of this fault tolerance enables the machine to keep operating, possibly at a reduced pace, rather than crashing completely.\vspace{0.7mm}

A RS $R_{f}$ is called fault-tolerant (FT) if $R_{f}\setminus{u^{i}}$ is also a RS, for all $u^{i}\in R_{f}$, and therefore the minimum cardinality of such $R_{f}$ is referred to as FTMD of $\Gamma$, denoted by $fdim(\Gamma)$. If $|R_{f}|=fdim(\Gamma)$, then $R_{f}$ is known as the $fault$-$tolerant$ $metric$ $basis$ (FTMB) for $\Gamma$. Fault-tolerant structures have been used successfully in engineering and computer science \cite{eck}. Slater initiated the study of fault-tolerant sets in \cite{sck}. In \cite{hn}, Hernando et al. proposed the idea of FTMD. They investigated the fault tolerance in trees and presented a significant result for FTMD in the form of an lower bound using MD, independent of graph choice, which is $dim(\Gamma)(1+2.5^{dim(\Gamma)-1})\leq fdim(\Gamma)$.\vspace{0.7mm}

The notion of FTMD is an interesting concept and has been studied by many researchers. The FTMD for $P_{n}$, $C_{n}$, and $K_{n}$ are as follows:\\\\
\textbf{Proposition 1.} {\it \cite{pck} For $n\geq3$, we have $fdim(P_{n})=2$, $fdim(C_{n})=3$, and $fdim(K_{n})=n$.}\\\\
In \cite{rz}, Raza et al. computed the FTMD of some classes of convex polytopes. Voronov in \cite{vv} investigated the FTMD of the king’s graph. For
more work on the FTMD, see references in \cite{acb, eck}. A subset $R^{i}$ of $V(\Gamma)$ is an independent set for $\Gamma$ if no two vertices in $R^{i}$ are adjacent.
The independence in RSs was first introduced by Chartrand et al. \cite{irs}. They characterize all connected non-trivial graphs $\Gamma$ of order $n$ with independent resolving numbers $1$, $n-2$, and $n-1$. \vspace{0.7mm}

Likewise resolving sets, in this work, we study the independence in FTRSs, and obtain it for some known graphs. We obtain the FTMD for three closely related classes of convex polytopes, viz., double antiprism $\mathbb{A}_{n}$, $S_{n}$, and $T_{n}$ \cite{ibs}. We locate FTRS of minimum cardinality in them. We conclude the article with some open problems regarding the independence of FTRSs. For the double antiprism $\mathbb{A}_{n}$, $S_{n}$, and $T_{n}$, Imran et al. in \cite{ibs}, proved the following:\\\\
\textbf{Proposition 2.}
{\it $dim(\mathbb{A}_{n})=dim(S_{n})=dim(T_{n})=3$, where $n\geq6$ is a positive integer.}\\

\hspace{-5.0mm}\textbf{{\Large $2$. Independent Fault-Tolerant Resolving Sets}}\\\\
Independent sets (or stable sets, ISs for short) in graphs are the most extensively studied concepts in graph theory. The $maximum$ $independent$ $sets$ (MISs) are ISs with maximum cardinality, and these ISs have received attention in the recent past. The $vertex$ $independence$ $number$ (or $independence$ $number$) of a graph $\Gamma$, denoted by $\beta(\Gamma)$, is the cardinality of MIS in $\Gamma$. There are also several ISs of minimum cardinality which are of interest with respect to theoretical and applied points of view.\vspace{0.7mm}

A $maximal$ $independent$ $set$ of vertices is an IS of vertices that are not properly contained in any other IS of vertices. The minimum cardinality of a maximal independent set is denoted by $i(\Gamma)$. This parameter is also known as the independent domination number because it has the smallest cardinality of an IS of vertices that dominates all the vertices of $\Gamma$. \vspace{0.7mm}

In \cite{irs}, Chartrand et al. explored the independence in resolving sets and provided some significant observations and results. We can see that some graphs consist of ISs $R^{i}_{f}$ with the property that $R^{i}_{f}-\{v^{i}\}$ is a RS for every $v^{i}$ in $R^{i}_{f}$. Therefore, this paper aims to detect the existence of such ISs in graphs and, if they exist, to study the minimum possible cardinality of such a set.\vspace{0.7mm}

An $independent$ $fault$-$tolerant$ $resolving$ $set$ (IFTRS) $R^{i}_{f}$ in a simple connected graph, $\Gamma$ is (1) independent and (2) FTRS. The cardinality of a minimum IFTRS in $\Gamma$ is the independent fault-tolerant resolving number, denoted by ifr$(\Gamma)$. Suppose $\Gamma$ be a non-trivial connected graph with $|V(\Gamma)|=n$, containing an IFTRS. Since every IFTRS is an FTRS, so it follows that
\begin{equation}\label{2}
 1+dim(\Gamma) \leq fdim(\Gamma) \leq ifr(\Gamma) \leq 1+\beta(\Gamma) \leq n
\end{equation}
\begin{center}
  \begin{figure}[h!]
  \centering
  \includegraphics[width=4.0in]{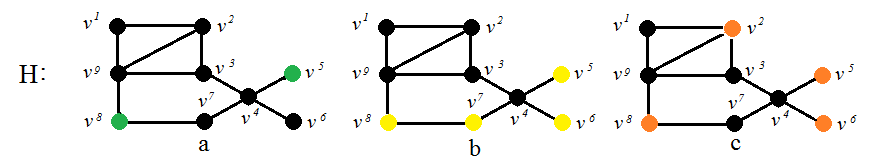}
  \caption{Graph $H$}\label{p4}
\end{figure}
\end{center}
To explain this concept, consider a graph $H$ in Fig. 1(a). The set $R=\{v^{5}, v^{8}\}$ (with green vertices) is a basis for $H$ and so $dim(H)=2$. Next, the set $R_{f}=\{v^{5}, v^{6}, v^{7}, v^{8}\}$ (with yellow vertices) in Fig. 1(b) is the FTRS set for $H$ and therefore $fdim(H)=4$. However, $R_{f}$ is not an IFTRS for $H$. The set $R^{i}_{f}=\{v^{2}, v^{5}, v^{6}, v^{8}\}$ (with orange vertices) in Fig. 1(c) is an IFTRS. Now, the co-ordinates of each vertex of $H$ with respect to $R^{i}_{f}$ are ${\ss}_{f}(v^{1}|R^{i}_{f})=(1,4,4,2)$, ${\ss}_{f}(v^{2}|R^{i}_{f})=(0,3,3,2)$, ${\ss}_{f}(v^{3}|R^{i}_{f})=(1,2,2,2)$, ${\ss}_{f}(v^{4}|R^{i}_{f})=(2,1,1,2)$, ${\ss}_{f}(v^{5}|R^{i}_{f})=(3,0,2,3)$, ${\ss}_{f}(v^{6}|R^{i}_{f})=(3,2,0,3)$, ${\ss}_{f}(v^{7}|R^{i}_{f})=(3,2,2,1)$, ${\ss}_{f}(v^{8}|R^{i}_{f})=(2,3,3,0)$, ${\ss}_{f}(v^{9}|R^{i}_{f})=(1,3,3,1)$. The codes with respect to the FTRS $R_{f}$ are called as the fault-tolerant metric codes (FTMC), denoted by ${\ss}_{f}(v^{j}|R_{f})$. A case-by-case analysis shows that $H$ contains no 3-element IFTRS (or FTRS) and so ifr$(H)=4$. The set $\{v^{1}, v^{3}, v^{5}, v^{6}, v^{7}\}$ is a MIS of $H$ and therefore $\beta(H)=5$. Thus the graph $H$ of Fig. 1 has $\beta(H)=5$, $dim(H)=2$, $fdim(H)=4$, and ifr$(H)=4$.\\

Likewise independent resolving set (IRS), all graphs do not have IFTRS, as a result, ifr$(\Gamma)$ is not defined for all graphs $\Gamma$. For example, the only ISs of the complete graph $K_{n}$; $n\geq 3$ are the singleton sets. Hence, ifr$(K_{n})$ is not defined for $n\geq 1$. In \cite{irs}, Chartrand et al. consider three regular graphs viz., the Petersen graph $P$, $K_{3,3}$, and $Q_{3}$ (see Fig. 2). For these graphs they found that, ir$(P)$=ir$(Q_{3})=3$ and for $K_{3,3}$, ir$(K_{3,3})$ does not exist.
\begin{center}
  \begin{figure}[h!]
  \centering
  \includegraphics[width=4.0in]{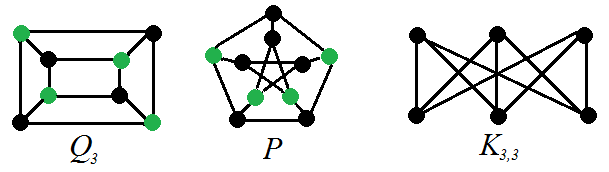}
  \caption{Regular Graphs}\label{p5}
\end{figure}
\end{center}

Similarly, for the IFTRS, we find that ifr$(K_{3,3})$ does not exist and ifr$(P)$=ifr$(Q_{3})=4$. The green vertices in Fig. 2 represent the minimum IFTRS for both $P$ and $Q_{3}$. Now, we have some results and observations regarding IRS and IFTRS.\\\\
\textbf{Proposition 3.} {\it $\Gamma=P_{n}$; $n\geq3$ iff  ifr$(P_{n})=2$.}
\begin{proof}
The proof is the same as for the FTMD of paths in \cite{hn}.
\end{proof}
\hspace{-5.0mm}Suppose $V(C_{n})= \{v^{1}, v^{2}, v^{3},...,v^{n}\}$ denotes the set of vertices in the cycle graph $C_{n}$. Then for IFTRS for $C_{n}$, we have:\\\\
\textbf{Proposition 4.} {\it For cycle graph $C_{n}$; $n\geq6$, we have ifr$(C_{n})=3$.}
\begin{proof}
Consider $R^{1}_{f}=\{v^{1}, v^{3}, v^{5}\}$ and $R^{2}_{f}=\{v^{1}, v^{3}, v^{6}\}$. Then, from Lemma 2 in \cite{cyc}, we find that, for $n\geq6$ and $n\neq8$, $R^{1}_{f}$ is the IFTRS for $C_{n}$ and $R^{2}_{f}$ is the IFTRS of $C_{n}$ for $n=8$. Therefore, ifr$(C_{n})=3$ for $n\geq6$.
\end{proof}
\hspace{-5.0mm}\textbf{Proposition 5.} {\it Every graph with IFTRS has IRS.}
\vspace{-5.0mm}
\begin{center}
  \begin{figure}[h!]
  \centering
  \includegraphics[width=3.8in]{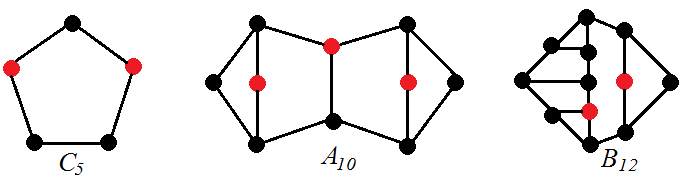}
  \caption{Graphs with IRS}\label{p5}
\end{figure}
\end{center}
By the definition of IFTRS, we see that proposition 5 is trivial. But the converse of the proposition 5 is not true. For example, suppose $C_{5}$, $A_{10}$, and $B_{12}$ are three graphs with $5$, $10$, and $12$ vertices, as shown in Fig. 3. We find that ir$(\Gamma)$ is defined for these three graphs (vertices in red color represent IRS), but ifr$(\Gamma)$ is not, where $\Gamma=C_{5}$, $A_{10}$, and $B_{12}$.\\

If $d(x,d)=d(x,c)$, $\forall$ $x\in V(\Gamma)- \{c,d\}$, then the vertices $d$ and $c$ are said to be $distance$ $similar$ (or distance equivalent) in $\Gamma$. Let $N(z)$ (open neighbourhood) be the set of vertices adjacent to $z$ in $\Gamma$, and let $N[z]=N(z)\cup \{z\}$ (closed neighbourhood). Then, in a non-trivial connected graph $\Gamma$, two vertices $x$ and $c$ are distance equivalent iff (1) $xc \in E(\Gamma)$ and $N[x]=N[c]$ or (2) $xc \notin E(\Gamma)$ and $N(x)=N(c)$. Moreover, the distance similarity is an equivalence relation on $V(\Gamma)$. Then we have the following observation.\\\\
\textbf{Observation 1.} {\it  In a connected graph $\Gamma$, if $D$ is a distance similar equivalence class with $|D|= w \geq 2$, then every FTRS of $\Gamma$ contains all the vertices from $D$.}\\

If $D$ is a distance similar equivalence class of $\Gamma$, then either the subgraph $\langle D\rangle$ induced by $D$ is complete in $\Gamma$ or $D$ is an independent set in $\Gamma$. Thus, we observe the following:\\\\
\textbf{Observation 2.} {\it Let $\Gamma$ be a graph and let $D$ with $|D| \geq 3$ be a distance similar equivalence class in $\Gamma$. Then ifr$(\Gamma)$ is not defined if $D$ is not independent in $\Gamma$.}\\

For observation 2, we find that the converse is not true. For instance, suppose $\Gamma=K_{3,3}$ with partite sets $A_{1}$ and $A_{2}$ (see Fig. 2). Then, we find that ifr$(\Gamma)$ is not defined. On the other side, $A_{1}$ and $A_{2}$ are the only two independent distance similar equivalence classes in $\Gamma=K_{3,3}$. Next, we determine the FTMD for three classes of convex polytopes.\\

\hspace{-5.0mm}\textbf{{\Large $3$. Minimum Fault-Tolerant Number of $\mathbb{A}_{n}$}}\\\\
In this section, we study some of the basic properties and the FTMD of $\mathbb{A}_{n}$.
The graph $\mathbb{A}_{n}$ consists of $3n$ vertices and $7n$ edges (see Fig. 4). The set of edges and vertices of $\mathbb{A}_{n}$ are denoted separately by $E(\mathbb{A}_{n})$ and $V(\mathbb{A}_{n})$, where $E(\mathbb{A}_{n})=\{j^{\bar{l}}_{r}j^{\bar{l}+1}_{r}, j^{\bar{l}}_{1}j^{\bar{l}}_{2}, j^{\bar{l}}_{2}j^{\bar{l}}_{3}, j^{\bar{l}}_{2}j^{\bar{l}+1}_{1}, j^{\bar{l}}_{3}j^{\bar{l}+1}_{2}:1 \leq r \leq 3; 1\leq \bar{l} \leq n\}$ and $V(\mathbb{A}_{n})=\{j^{\bar{l}}_{r}:1 \leq r \leq 3;1 \leq \bar{l} \leq n\}$.
\begin{center}
  \begin{figure}[h!]
  \centering
  \includegraphics[width=2.5in]{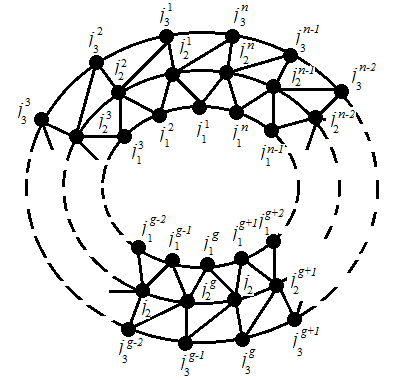}
  \caption{Double Antiprism $\mathbb{A}_{n}$}\label{p1}
\end{figure}
\end{center}
We call vertices $\{j^{\bar{l}}_{1}:1 \leq \bar{l} \leq n\}$, $\{j^{\bar{l}}_{2}:1 \leq \bar{l} \leq n\}$, and $\{j^{\bar{l}}_{3}:1 \leq \bar{l} \leq n\}$ as the inner, middle and outer cycle vertices in $\mathbb{A}_{n}$, respectively.  In the following result, we investigate the FTMD of $\mathbb{A}_{n}$.

\begin{thm}
$fdim(\mathbb{A}_{n})=4$, where $n\geq6$ is a positive integer.
\end{thm}

\begin{proof}
To show that $fdim(\mathbb{A}_{n})\leq 4$, we consider the following two cases i.e., $n\equiv 0(mod\ 2)$ and $n\equiv 1(mod\ 2)$.\\
\textbf{Case(\rom{1})} $n\equiv 0(mod\ 2)$\\
For this, we set $n=2\bar{g}$, $\bar{g} \in \mathbb{N}$, and $\bar{g}\geq3$. Suppose $R_{f}=\{j^{1}_{1}, j^{2}_{1}, j^{\bar{g}+1}_{1}, j^{\bar{g}+2}_{1}\} \subseteq V(\mathbb{A}_{n})$. Next, we give FTMC to every vertex of $\mathbb{A}_{n}$ with respect to $R_{f}$. \\
For the vertices $\{j^{\bar{l}}_{1}:1\leq \bar{l} \leq n\}$, the FTMC are
\begin{center}
 \begin{tabular}{|m{14.0em}|c|c|c|c|}
 \hline
 ${\ss}_{f}(j^{\bar{l}}_{1}|R_{f})$                            & $j^{1}_{1}$ & $j^{2}_{1}$  & $j^{\bar{g}+1}_{1}$   & $j^{\bar{g}+2}_{1}$\\
 \hline
 ${\ss}_{f}(j^{\bar{l}}_{1}|R_{f})$:($\bar{l}=1$)              & $\bar{l}-1$    & $1$       & $\bar{g}-\bar{l}+1$   & $\bar{g}-1$ \\
 \hline
 ${\ss}_{f}(j^{\bar{l}}_{1}|R_{f})$:($2\leq \bar{l}\leq \bar{g}+1$)& $\bar{l}-1$ & $\bar{l}-2$ & $\bar{g}-\bar{l}+1$ & $\bar{g}-\bar{l}+2$  \\
 \hline
 ${\ss}_{f}(j^{\bar{l}}_{1}|R_{f})$:($\bar{l}=\bar{g}+2$)       & $2\bar{g}-\bar{l}+1$ & $\bar{l}-2$& $\bar{l}-\bar{g}-1$  & $\bar{g}-\bar{l}+2$  \\
 \hline
 ${\ss}_{f}(j^{\bar{l}}_{1}|R_{f})$:($\bar{g}+3\leq \bar{l}\leq 2\bar{g}$) & $2\bar{g}-\bar{l}+1$ & $2\bar{g}-\bar{l}+2$ & $\bar{l}-\bar{g}-1$   & $\bar{l}-\bar{g}-2$  \\
 \hline
 \end{tabular}
 \end{center}
For the vertices $\{j^{\bar{l}}_{2}:1\leq \bar{l} \leq n\}$, the FTMC are
\begin{center}
 \begin{tabular}{|m{14.0em}|c|c|c|c|}
 \hline
 ${\ss}_{f}(j^{\bar{l}}_{2}|R_{f})$                              & $j^{1}_{1}$  & $j^{2}_{1}$ & $j^{\bar{g}+1}_{1}$ & $j^{\bar{g}+2}_{1}$\\
 \hline
 ${\ss}_{f}(j^{\bar{l}}_{2}|R_{f})$:($\bar{l}=1$)                & $\bar{l}$    & $1$      & $\bar{g}-\bar{l}+1$    & $\bar{g}$ \\
 \hline
 ${\ss}_{f}(j^{\bar{l}}_{2}|R_{f})$:($2\leq \bar{l}\leq \bar{g}$) & $\bar{l}$    & $\bar{l}-1$    & $\bar{g}-\bar{l}+1$   & $\bar{g}-\bar{l}+2$  \\
 \hline
 ${\ss}_{f}(j^{\bar{l}}_{2}|R_{f})$:($\bar{l}=\bar{g}+1$)         & $2\bar{g}-\bar{l}+1$ & $\bar{l}-1$ & $\bar{l}-\bar{g}$ & $\bar{g}-\bar{l}+2$  \\
 \hline
 ${\ss}_{f}(j^{\bar{l}}_{2}|R_{f})$:($\bar{g}+2\leq \bar{l}\leq 2\bar{g}$)   & $2\bar{g}-\bar{l}+1$ & $2\bar{g}-\bar{l}+2$ & $\bar{l}-\bar{g}$   & $\bar{l}-\bar{g}-1$  \\
 \hline
 \end{tabular}
 \end{center}
Finally, for the vertices $\{j^{\bar{l}}_{3}:1\leq \bar{l} \leq n\}$, the FTMC are
\begin{center}
 \begin{tabular}{|m{14.0em}|c|c|c|c|}
 \hline
  ${\ss}_{f}(j^{\bar{l}}_{3}|R_{f})$                             & $j^{1}_{1}$    & $j^{2}_{1}$  & $j^{\bar{g}+1}_{1}$   & $j^{\bar{g}+2}_{1}$\\
 \hline
 ${\ss}_{f}(j^{\bar{l}}_{3}|R_{f})$:($\bar{l}=1$)                & $\bar{l}+1$    & $2$       & $\bar{g}-\bar{l}+1$   & $\bar{g}-\bar{l}+2$ \\
 \hline
 ${\ss}_{f}(j^{\bar{l}}_{3}|R_{f})$:($2\leq \bar{l}\leq \bar{g}-1$) & $\bar{l}+1$ & $\bar{l}$ & $\bar{g}-\bar{l}+1$   & $\bar{g}-\bar{l}+2$  \\
 \hline
 ${\ss}_{f}(j^{\bar{l}}_{3}|R_{f})$:($\bar{l}=\bar{g}$)          & $\bar{l}+1$     & $\bar{l}$      & $2$             & $\bar{g}-\bar{l}+2$ \\
 \hline
 ${\ss}_{f}(j^{\bar{l}}_{3}|R_{f})$:($\bar{l}=\bar{g}+1$)        & $2\bar{g}-\bar{l}+1$  & $\bar{l}+2$    & $\bar{l}-\bar{g}+1$   & $2$  \\
 \hline
 ${\ss}_{f}(j^{\bar{l}}_{3}|R_{f})$:($\bar{g}+2\leq \bar{l}\leq 2\bar{g}-1$) & $2\bar{g}-\bar{l}+1$  & $2\bar{g}-\bar{l}+2$ & $\bar{l}-\bar{g}+1$   & $\bar{l}-\bar{g}$  \\
 \hline
 ${\ss}_{f}(j^{\bar{l}}_{3}|R_{f})$:($\bar{l}=2\bar{g}$)           & $2$   & $2\bar{g}-\bar{l}+2$ & $\bar{l}-\bar{g}+1$   & $\bar{l}-\bar{g}$  \\
 \hline
 \end{tabular}
 \end{center}
From these codes, we see that no two elements in $V(\mathbb{A}_{n})$ have the same FTMC, implying $R_{f}$ to be a resolving set for $\mathbb{A}_{n}$. Since, by definition of FTRS, the subsets $R_{f}\smallsetminus \{j\}$, $\forall$ $j\in R_{f}$ are $R_{1}=\{j^{1}_{1}, j^{2}_{1}, j^{\bar{g}+1}_{1}\}$, $R_{2}=\{j^{1}_{1}, j^{2}_{1}, j^{\bar{g}+2}_{1}\}$, $R_{3}=\{j^{1}_{1}, j^{\bar{g}+1}_{1}, j^{\bar{g}+2}_{1}\}$, and $R_{4}=\{j^{2}_{1}, j^{\bar{g}+1}_{1}, j^{\bar{g}+2}_{1}\}$. To prove that the set $R_{f}$ is the FTRS for $\mathbb{A}_{n}$, we have to show that the sets $R_{1}$, $R_{2}$, $R_{3}$, and $R_{4}$ are the resolving sets for $\mathbb{A}_{n}$. Then from the FTMC, as shown above, we find that the sets $R_{1}$, $R_{2}$, $R_{3}$, and $R_{4}$ are also resolving sets for $\mathbb{A}_{n}$, as the metric coordinates for every distinct pair of vertices of $\mathbb{A}_{n}$ are different with respect to the sets $R_{1}$, $R_{2}$, $R_{3}$, and $R_{4}$. This implies $fdim(\mathbb{A}_{n})\leq 4$. Thus, from these lines, equation (1), and proposition 2 we have $fdim(\mathbb{A}_{n})= 4$, in this case.\\
\textbf{Case(\rom{2})} $n\equiv 1(mod\ 2)$\\
For this, we set $n=2\bar{g}+1$, $\bar{g} \in \mathbb{N}$, and $\bar{g}\geq3$. Suppose $R_{f}=\{j^{1}_{1}, j^{2}_{1}, j^{\bar{g}+1}_{1}, j^{\bar{g}+3}_{1}\}\subseteq V(\mathbb{A}_{n})$. Next, we give FTMC to every vertex of $\mathbb{A}_{n}$ with respect to $R_{f}$. \\
For the vertices $\{j^{\bar{l}}_{1}:1\leq \bar{l} \leq n\}$, the FTMC are
\begin{center}
 \begin{tabular}{|m{14.0em}|c|c|c|c|}
 \hline
 ${\ss}_{f}(j^{\bar{l}}_{1}|R_{f})$                            & $j^{1}_{1}$    & $j^{2}_{1}$    & $j^{\bar{g}+1}_{1}$   & $j^{\bar{g}+3}_{1}$\\
 \hline
 ${\ss}_{f}(j^{\bar{l}}_{1}|R_{f})$:($\bar{l}=1$)              & $\bar{l}-1$    & $1$            & $\bar{g}-\bar{l}+1$   & $\bar{g}-1$ \\
 \hline
 ${\ss}_{f}(j^{\bar{l}}_{1}|R_{f})$:($\bar{l}=2$)              & $\bar{l}-1$    & $\bar{l}-2$    & $\bar{g}-\bar{l}+1$   & $\bar{g}$ \\
 \hline
 ${\ss}_{f}(j^{\bar{l}}_{1}|R_{f})$:($3\leq \bar{l}\leq \bar{g}+1$)& $\bar{l}-1$ & $\bar{l}-2$    & $\bar{g}-\bar{l}+1$   & $\bar{g}-\bar{l}+3$  \\
 \hline
 ${\ss}_{f}(j^{\bar{l}}_{1}|R_{f})$:($\bar{l}=\bar{g}+2$) & $2\bar{g}-\bar{l}+2$ & $\bar{l}-2$ & $\bar{l}-\bar{g}-1$ & $\bar{g}-\bar{l}+3$  \\
 \hline
 ${\ss}_{f}(j^{\bar{l}}_{1}|R_{f})$:($\bar{l}=\bar{g}+3$) & $2\bar{g}-\bar{l}+2$& $2\bar{g}-\bar{l}+3$ & $\bar{l}-\bar{g}-1$& $\bar{g}-\bar{l}+3$  \\
 \hline
 ${\ss}_{f}(j^{\bar{l}}_{1}|R_{f})$:($\bar{g}+4\leq \bar{l}\leq 2\bar{g}+1$) & $2\bar{g}-\bar{l}+2$ & $2\bar{g}-\bar{l}+3$ & $\bar{l}-\bar{g}-1$   & $\bar{l}-\bar{g}-3$  \\
 \hline
 \end{tabular}
 \end{center}
For the vertices $\{j^{\bar{l}}_{2}:1\leq \bar{l} \leq n\}$, the FTMC are
\begin{center}
 \begin{tabular}{|m{14.0em}|c|c|c|c|}
 \hline
 ${\ss}_{f}(j^{\bar{l}}_{2}|R_{f})$                              & $j^{1}_{1}$    & $j^{2}_{1}$    & $j^{\bar{g}+1}_{1}$ & $j^{\bar{g}+3}_{1}$\\
 \hline
 ${\ss}_{f}(j^{\bar{l}}_{2}|R_{f})$:($\bar{l}=1$)                & $\bar{l}$      & $1$            & $\bar{g}-\bar{l}+1$ & $\bar{g}$ \\
 \hline
 ${\ss}_{f}(j^{\bar{l}}_{2}|R_{f})$:($2\leq \bar{l}\leq \bar{g}$)& $\bar{l}$      & $\bar{l}-1$    & $\bar{g}-\bar{l}+1$ & $\bar{g}-\bar{l}+3$  \\
 \hline
 ${\ss}_{f}(j^{\bar{l}}_{2}|R_{f})$:($\bar{l}=\bar{g}+1$)        & $2\bar{g}-\bar{l}+2$ & $\bar{l}-1$& $\bar{l}-\bar{g}$   & $\bar{g}-\bar{l}+3$  \\
 \hline
 ${\ss}_{f}(j^{\bar{l}}_{2}|R_{f})$:($\bar{l}=\bar{g}+2$)  & $2\bar{g}-\bar{l}+2$ & $2\bar{g}-\bar{l}+3$ & $\bar{l}-\bar{g}$& $\bar{g}-\bar{l}+3$  \\
 \hline
 ${\ss}_{f}(j^{\bar{l}}_{2}|R_{f})$:($\bar{g}+3\leq \bar{l}\leq 2\bar{g}+1$) & $2\bar{g}-\bar{l}+2$ & $2\bar{g}-\bar{l}+3$ & $\bar{l}-\bar{g}$   & $\bar{l}-\bar{g}-2$  \\
 \hline
 \end{tabular}
 \end{center}
Finally, for the vertices $\{j^{\bar{l}}_{3}:1\leq \bar{l} \leq n\}$, the FTMC are
\begin{center}
 \begin{tabular}{|m{14.0em}|c|c|c|c|}
 \hline
  ${\ss}_{f}(j^{\bar{l}}_{3}|R_{f})$               & $j^{1}_{1}$     & $j^{2}_{1}$ & $j^{\bar{g}+1}_{1}$   & $j^{\bar{g}+3}_{1}$\\
 \hline
 ${\ss}_{f}(j^{\bar{l}}_{3}|R_{f})$:($\bar{l}=1$)  & $\bar{l}+1$     & $2$         & $\bar{g}-\bar{l}+1$   & $\bar{g}+1$ \\
 \hline
 ${\ss}_{f}(j^{\bar{l}}_{3}|R_{f})$:($2\leq \bar{l}\leq \bar{g}-1$)  & $\bar{l}+1$ & $\bar{l}$  & $\bar{g}-\bar{l}+1$ & $\bar{g}-\bar{l}+3$  \\
 \hline
 ${\ss}_{f}(j^{\bar{l}}_{3}|R_{f})$:($\bar{l}=\bar{g}$)              & $\bar{l}+1$     & $\bar{l}$      & $2$             & $\bar{g}-\bar{l}+3$ \\
 \hline
 ${\ss}_{f}(j^{\bar{l}}_{3}|R_{f})$:($\bar{l}=\bar{g}+1$)            & $2\bar{g}-\bar{l}+2$  & $\bar{l}$ & $\bar{l}-\bar{g}+1$   & $\bar{g}-\bar{l}+3$  \\
 \hline
 ${\ss}_{f}(j^{\bar{l}}_{3}|R_{f})$:($\bar{l}=\bar{g}+1$)            & $2\bar{g}-\bar{l}+2$  & $2\bar{g}-\bar{l}+3$ & $\bar{l}-\bar{g}+1$   & $2$  \\
 \hline
 ${\ss}_{f}(j^{\bar{l}}_{3}|R_{f})$:($\bar{g}+2\leq \bar{l}\leq 2\bar{g}$)& $2\bar{g}-\bar{l}+2$  & $2\bar{g}-\bar{l}+3$ & $\bar{l}-\bar{g}+1$   & $\bar{l}-\bar{g}-1$  \\
 \hline
 ${\ss}_{f}(j^{\bar{l}}_{3}|R_{f})$:($\bar{l}=2\bar{g}+1$)             & $2$ & $2\bar{g}-\bar{l}+3$ & $\bar{g}+1$ & $\bar{l}-\bar{g}-1$  \\
 \hline
 \end{tabular}
 \end{center}
From these codes, we see that no two elements in $V(\mathbb{A}_{n})$ have the same FTMC, implying $R_{f}$ to be a resolving set for $\mathbb{A}_{n}$. Since, by definition of FTRS, the subsets $R_{f}\smallsetminus \{j\}$, $\forall$ $j\in R_{f}$ are $R_{1}=\{j^{1}_{1}, j^{2}_{1}, j^{\bar{g}+1}_{1}\}$, $R_{2}=\{j^{1}_{1}, j^{2}_{1}, j^{\bar{g}+3}_{1}\}$, $R_{3}=\{j^{1}_{1}, j^{\bar{g}+1}_{1}, j^{\bar{g}+3}_{1}\}$, and $R_{4}=\{j^{2}_{1}, j^{\bar{g}+1}_{1}, j^{\bar{g}+3}_{1}\}$. To prove that the set $R_{f}$ is the FTRS for $\mathbb{A}_{n}$, we have to show that the sets $R_{1}$, $R_{2}$, $R_{3}$, and $R_{4}$ are the resolving sets for $\mathbb{A}_{n}$. Then from the FTMC, as shown above, we find that the sets $R_{1}$, $R_{2}$, $R_{3}$, and $R_{4}$ are also resolving sets for $\mathbb{A}_{n}$, as the metric coordinates for every distinct pair of vertices of $\mathbb{A}_{n}$ are different with respect to the sets $R_{1}$, $R_{2}$, $R_{3}$, and $R_{4}$. This implies $fdim(\mathbb{A}_{n})\leq 4$. Thus, from these lines, equation (1), and proposition 2 we have $fdim(\mathbb{A}_{n})= 4$, as well in this case.
\end{proof}

\begin{cor}
The FTMD for the double antiprism $\mathbb{A}_{n}$ is constant.
\end{cor}

\hspace{-5.0mm}\textbf{{\Large $4$. Minimum Fault-Tolerant Number of $S_{n}$}}\\\\
In this section, we study some of the basic properties and the FTMD of $S_{n}$.
The graph $S_{n}$ consists of $4n$ vertices and $8n$ edges (see Fig. 5). The set of edges and vertices of $S_{n}$ are denoted separately by $E(S_{n})$ and $V(S_{n})$, where $E(S_{n})=\{j^{\bar{l}}_{r}j^{\bar{l+1}}_{r}, j^{\bar{l}}_{1}j^{\bar{l}}_{2}, j^{\bar{l}}_{2}j^{\bar{l}}_{3}, j^{\bar{l}}_{3}j^{\bar{l}}_{4}, j^{\bar{l}}_{3}j^{\bar{l}+1}_{2}:1 \leq r \leq 4; 1\leq \bar{l} \leq n\}$ and $V(S_{n})=\{j^{\bar{l}}_{r}:1 \leq r \leq 4;1 \leq \bar{l} \leq n\}$.
\begin{center}
  \begin{figure}[h!]
  \centering
  \includegraphics[width=2.5in]{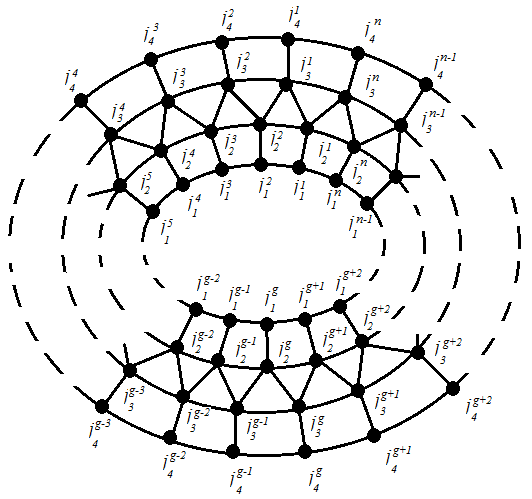}
  \caption{The Graph $S_{n}$}\label{p2}
\end{figure}
\end{center}
We call vertices $\{j^{\bar{l}}_{1}:1 \leq \bar{l} \leq n\}$, $\{j^{\bar{l}}_{2}:1 \leq \bar{l} \leq n\}$, $\{j^{\bar{l}}_{3}:1 \leq \bar{l} \leq n\}$, and $\{j^{\bar{l}}_{4}:1 \leq \bar{l} \leq n\}$ as the first, second, third, and fourth cycle vertices in $S_{n}$, respectively.  In the following result, we investigate the FTMD of $S_{n}$.

\begin{thm}
$fdim(S_{n})=4$, where $n\geq6$ is a positive integer.
\end{thm}

\begin{proof}
To show that $fdim(S_{n})\leq 4$, we consider the following two cases i.e., $n\equiv 0(mod\ 2)$ and $n\equiv 1(mod\ 2)$.\\
\textbf{Case(\rom{1})} $n\equiv 0(mod\ 2)$\\
For this, we set $n=2\bar{g}$, $\bar{g} \in \mathbb{N}$ and $\bar{g}\geq3$. Suppose $R_{f}=\{j^{1}_{1}, j^{2}_{1}, j^{\bar{g}+1}_{1}, j^{\bar{g}+2}_{1}\} \subseteq V(S_{n})$. Next, we give FTMC to every vertex of $S_{n}$ with respect to $R_{f}$. \\
For the vertices $\{j^{\bar{l}}_{1}:1\leq \bar{l} \leq n\}$, the FTMC are
\begin{center}
 \begin{tabular}{|m{14.0em}|c|c|c|c|}
 \hline
 ${\ss}_{f}(j^{\bar{l}}_{1}|R_{f})$                            & $j^{1}_{1}$    & $j^{2}_{1}$    & $j^{\bar{g}+1}_{1}$   & $j^{\bar{g}+2}_{1}$\\
 \hline
 ${\ss}_{f}(j^{\bar{l}}_{1}|R_{f})$:($\bar{l}=1$)              & $\bar{l}-1$    & $1$            & $\bar{g}-\bar{l}+1$   & $\bar{g}-1$ \\
 \hline
 ${\ss}_{f}(j^{\bar{l}}_{1}|R_{f})$:($2\leq \bar{l}\leq \bar{g}+1$)  & $\bar{l}-1$ & $\bar{l}-2$ & $\bar{g}-\bar{l}+1$   & $\bar{g}-\bar{l}+2$  \\
 \hline
 ${\ss}_{f}(j^{\bar{l}}_{1}|R_{f})$                        & $j^{1}_{1}$    & $j^{2}_{1}$    & $j^{\bar{g}+1}_{1}$   & $j^{\bar{g}+2}_{1}$\\
 \hline
 ${\ss}_{f}(j^{\bar{l}}_{1}|R_{f})$:($\bar{l}=\bar{g}+2$)  & $2\bar{g}-\bar{l}+1$ & $\bar{l}-2$    & $\bar{l}-\bar{g}-1$   & $\bar{g}-\bar{l}+2$  \\
 \hline
 ${\ss}_{f}(j^{\bar{l}}_{1}|R_{f})$:($\bar{g}+3\leq \bar{l}\leq 2\bar{g}$) & $2\bar{g}-\bar{l}+1$ & $2\bar{g}-\bar{l}+2$ & $\bar{l}-\bar{g}-1$   & $\bar{l}-\bar{g}-2$  \\
 \hline
 \end{tabular}
 \end{center}
For the vertices $\{j^{\bar{l}}_{2}:1\leq \bar{l} \leq n\}$, the FTMC are ${\ss}_{f}(j^{\bar{l}}_{2}|R_{f})={\ss}_{f}(j^{\bar{l}}_{1}|R_{f})+(1,1,1,1)$, for $1\leq \bar{l} \leq n$. Next, for the vertices $\{j^{\bar{l}}_{3}:1\leq \bar{l} \leq n\}$, the FTMC are
\begin{center}
 \begin{tabular}{|m{14.0em}|c|c|c|c|}
 \hline
  ${\ss}_{f}(j^{\bar{l}}_{3}|R_{f})$               & $j^{1}_{1}$     & $j^{2}_{1}$    & $j^{\bar{g}+1}_{1}$   & $j^{\bar{g}+2}_{1}$\\
 \hline
 ${\ss}_{f}(j^{\bar{l}}_{3}|R_{f})$:($\bar{l}=1$)  & $\bar{l}+1$     & $2$            & $\bar{g}-\bar{l}+2$   & $\bar{g}+1$ \\
 \hline
 ${\ss}_{f}(j^{\bar{l}}_{3}|R_{f})$:($2\leq \bar{l}\leq \bar{g}$)    & $\bar{l}+1$    & $\bar{l}$  & $\bar{g}-\bar{l}+2$   & $\bar{g}-\bar{l}+3$  \\
 \hline
 ${\ss}_{f}(j^{\bar{l}}_{3}|R_{f})$:($\bar{l}=\bar{g}+1$)  & $2\bar{g}-\bar{l}+2$  & $\bar{l}$  & $\bar{l}-\bar{g}+1$   & $\bar{g}-\bar{l}+3$  \\
 \hline
 ${\ss}_{f}(j^{\bar{l}}_{3}|R_{f})$:($\bar{g}+2\leq \bar{l}\leq 2\bar{g}$)   & $2\bar{g}-\bar{l}+2$  & $2\bar{g}-\bar{l}+3$ & $\bar{l}-\bar{g}+1$   & $\bar{l}-\bar{g}$  \\
 \hline
\end{tabular}
\end{center}
Finally, for the vertices $\{j^{\bar{l}}_{4}:1\leq \bar{l} \leq n\}$, the FTMC are ${\ss}_{f}(j^{\bar{l}}_{4}|R_{f})={\ss}_{f}(j^{\bar{l}}_{3}|R_{f})+(1,1,1,1)$, for $1\leq \bar{l} \leq n$. From these codes, we see that no two elements in $V(S_{n})$ have the same FTMC, implying $R_{f}$ to be a resolving set for $S_{n}$. Since, by definition of FTRS, the subsets $R_{f}\smallsetminus \{j\}$, $\forall$ $j\in R_{f}$ are $R_{1}=\{j^{1}_{1}, j^{2}_{1}, j^{\bar{g}+1}_{1}\}$, $R_{2}=\{j^{1}_{1}, j^{2}_{1}, j^{\bar{g}+2}_{1}\}$, $R_{3}=\{j^{1}_{1}, j^{\bar{g}+1}_{1}, j^{\bar{g}+2}_{1}\}$, and $R_{4}=\{j^{2}_{1}, j^{\bar{g}+1}_{1}, j^{\bar{g}+2}_{1}\}$. To prove that the set $R_{f}$ is the FTRS for $S_{n}$, we have to show that the sets $R_{1}$, $R_{2}$, $R_{3}$, and $R_{4}$ are the resolving sets for $S_{n}$. Then from the FTMC, as shown above, we find that the sets $R_{1}$, $R_{2}$, $R_{3}$, and $R_{4}$ are also resolving sets for $S_{n}$, as the metric coordinates for every distinct pair of vertices of $S_{n}$ are different with respect to the sets $R_{1}$, $R_{2}$, $R_{3}$, and $R_{4}$. This implies $fdim(S_{n})\leq 4$. Thus, from these lines, equation (1), and proposition 2 we have $fdim(S_{n})= 4$, in this case.\\
\textbf{Case(\rom{2})} $n\equiv 1(mod\ 2)$\\
For this, we set $n=2\bar{g}+1$, $\bar{g} \in \mathbb{N}$ and $\bar{g}\geq3$. Suppose $R_{f}=\{j^{1}_{1}, j^{2}_{1}, j^{\bar{g}+1}_{1}, j^{\bar{g}+2}_{1}\}\subseteq V(S_{n})$. Next, we give FTMC to every vertex of $S_{n}$ with respect to $R_{f}$. \\
For the vertices $\{j^{\bar{l}}_{1}:1\leq \bar{l} \leq n\}$, the FTMC are
\begin{center}
 \begin{tabular}{|m{14.0em}|c|c|c|c|}
 \hline
 ${\ss}_{f}(j^{\bar{l}}_{1}|R_{f})$                           & $j^{1}_{1}$  & $j^{2}_{1}$ & $j^{\bar{g}+1}_{1}$   & $j^{\bar{g}+2}_{1}$\\
 \hline
 ${\ss}_{f}(j^{\bar{l}}_{1}|R_{f})$:($\bar{l}=1$)             & $\bar{l}-1$  & $1$         & $\bar{g}-\bar{l}+1$   & $\bar{g}$ \\
 \hline
 ${\ss}_{f}(j^{\bar{l}}_{1}|R_{f})$:($2\leq \bar{l}\leq \bar{g}+1$)  & $\bar{l}-1$    & $\bar{l}-2$    & $\bar{g}-\bar{l}+1$   & $\bar{g}-\bar{l}+2$  \\
 \hline
 ${\ss}_{f}(j^{\bar{l}}_{1}|R_{f})$:($\bar{l}=\bar{g}+2$)   & $2\bar{g}-\bar{l}+2$ & $\bar{l}-2$    & $\bar{l}-\bar{g}-1$   & $\bar{g}-\bar{l}+2$  \\
 \hline
 ${\ss}_{f}(j^{\bar{l}}_{1}|R_{f})$:($\bar{g}+3\leq \bar{l}\leq 2\bar{g}+1$) & $2\bar{g}-\bar{l}+2$ & $2\bar{g}-\bar{l}+3$ & $\bar{l}-\bar{g}-1$   & $\bar{l}-\bar{g}-2$  \\
 \hline
 \end{tabular}
 \end{center}
For the vertices $\{j^{\bar{l}}_{2}:1\leq \bar{l} \leq n\}$, the FTMC are ${\ss}_{f}(j^{\bar{l}}_{2}|R_{f})={\ss}_{f}(j^{\bar{l}}_{1}|R_{f})+(1,1,1,1)$, for $1\leq \bar{l} \leq n$. Next, for the vertices $\{j^{\bar{l}}_{3}:1\leq \bar{l} \leq n\}$, the FTMC are
\begin{center}
 \begin{tabular}{|m{14.0em}|c|c|c|c|}
 \hline
  ${\ss}_{f}(j^{\bar{l}}_{3}|R_{f})$                & $j^{1}_{1}$     & $j^{2}_{1}$    & $j^{\bar{g}+1}_{1}$   & $j^{\bar{g}+2}_{1}$\\
 \hline
 ${\ss}_{f}(j^{\bar{l}}_{3}|R_{f})$:($\bar{l}=1$)   & $\bar{l}+1$     & $2$            & $\bar{g}-\bar{l}+2$   & $\bar{g}-\bar{l}+3$ \\
 \hline
 ${\ss}_{f}(j^{\bar{l}}_{3}|R_{f})$:($2\leq \bar{l}\leq \bar{g}$)     & $\bar{l}+1$    & $\bar{l}$   & $\bar{g}-\bar{l}+2$ & $\bar{g}-\bar{l}+3$  \\
 \hline
 ${\ss}_{f}(j^{\bar{l}}_{3}|R_{f})$:($\bar{l}=\bar{g}+1$)  & $2\bar{g}-\bar{l}+3$  & $\bar{l}$    & $\bar{l}-\bar{g}+1$   & $\bar{g}-\bar{l}+3$  \\
 \hline
 ${\ss}_{f}(j^{\bar{l}}_{3}|R_{f})$:($\bar{g}+2\leq \bar{l}\leq 2\bar{g}+1$) & $2\bar{g}-\bar{l}+3$  & $2\bar{g}-\bar{l}+4$ & $\bar{l}-\bar{g}+1$   & $\bar{l}-\bar{g}$  \\
 \hline
\end{tabular}
\end{center}
Finally, for the vertices $\{j^{\bar{l}}_{4}:1\leq \bar{l} \leq n\}$, the FTMC are ${\ss}_{f}(j^{\bar{l}}_{4}|R_{f})={\ss}_{f}(j^{\bar{l}}_{3}|R_{f})+(1,1,1,1)$, for $1\leq \bar{l} \leq n$. From these codes, we see that no two elements in $V(S_{n})$ have the same FTMC, implying $R_{f}$ to be a resolving set for $S_{n}$. Since, by definition of FTRS, the subsets $R_{f}\smallsetminus \{j\}$, $\forall$ $j\in R_{f}$ are $R_{1}=\{j^{1}_{1}, j^{2}_{1}, j^{\bar{g}+1}_{1}\}$, $R_{2}=\{j^{1}_{1}, j^{2}_{1}, j^{\bar{g}+2}_{1}\}$, $R_{3}=\{j^{1}_{1}, j^{\bar{g}+1}_{1}, j^{\bar{g}+2}_{1}\}$, and $R_{4}=\{j^{2}_{1}, j^{\bar{g}+1}_{1}, j^{\bar{g}+2}_{1}\}$. To prove that the set $R_{f}$ is the FTRS for $S_{n}$, we have to show that the sets $R_{1}$, $R_{2}$, $R_{3}$, and $R_{4}$ are the resolving sets for $S_{n}$. Then from the FTMC, as shown above, we find that the sets $R_{1}$, $R_{2}$, $R_{3}$, and $R_{4}$ are also resolving sets for $S_{n}$, as the metric coordinates for every distinct pair of vertices of $S_{n}$ are different with respect to the sets $R_{1}$, $R_{2}$, $R_{3}$, and $R_{4}$. This implies $fdim(S_{n})\leq 4$. Thus, from these lines, equation (1), and proposition 2 we have $fdim(S_{n})= 4$, as well in this case.
\end{proof}

\begin{cor}
The FTMD for $S_{n}$ is constant.
\end{cor}

\hspace{-5.0mm}\textbf{{\Large $5$. Minimum Fault-Tolerant Number of $T_{n}$}}\\\\
In this section, we study some of the basic properties and the FTMD of $T_{n}$.
The graph $T_{n}$ consists of $4n$ vertices and $8n$ edges (see Fig. 6). The set of edges and vertices of $T_{n}$ are denoted separately by $E(T_{n})$ and $V(T_{n})$, where $E(T_{n})=\{j^{\bar{l}}_{r}j^{\bar{l+1}}_{r}, j^{\bar{l}}_{1}j^{\bar{l}}_{2}, j^{\bar{l}}_{2}j^{\bar{l}}_{3}, j^{\bar{l}}_{3}j^{\bar{l}}_{4}, j^{\bar{l}}_{2}j^{\bar{l}+1}_{1}, j^{\bar{l}}_{3}j^{\bar{l}+1}_{2}:r=1,2,4; 1\leq \bar{l} \leq n\}$ and $V(T_{n})=\{j^{\bar{l}}_{r}:1 \leq r \leq 4;1 \leq \bar{l} \leq n\}$.
\begin{center}
  \begin{figure}[h!]
  \centering
  \includegraphics[width=2.5in]{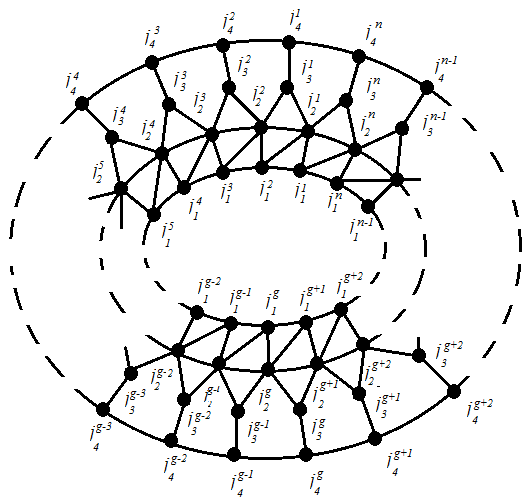}
  \caption{The Graph $T_{n}$}\label{p3}
\end{figure}
\end{center}
We call vertices $\{j^{\bar{l}}_{1}: 1 \leq \bar{l} \leq n\}$, $\{j^{\bar{l}}_{2}: 1 \leq \bar{l} \leq n\}$, and $\{j^{\bar{l}}_{4}:1 \leq \bar{l} \leq n\}$ as the first, second and fourth cycle vertices in $T_{n}$, respectively. The remaining vertices $\{j^{\bar{l}}_{3}: 1 \leq \bar{l} \leq n\}$, as the middle vertices in $T_{n}$. In the following result, we investigate the FTMD of $T_{n}$.

\begin{thm}
$fdim(T_{n})=4$, where $n\geq6$ is a positive integer.
\end{thm}

\begin{proof}
To show that $fdim(T_{n})\leq 4$, we consider the following two cases i.e., $n\equiv 0(mod\ 2)$ and $n\equiv 1(mod\ 2)$.\\
\textbf{Case(\rom{1})} $n\equiv 0(mod\ 2)$\\
For this, we set $n=2\bar{g}$, $\bar{g} \in \mathbb{N}$ and $\bar{g}\geq3$. Suppose $R_{f}=\{j^{1}_{1}, j^{2}_{1}, j^{\bar{g}+1}_{1}, j^{\bar{g}+2}_{1}\} \subseteq V(T_{n})$. Next, we give FTMC to every vertex of $T_{n}$ with respect to $R_{f}$. \\
For the vertices $\{j^{\bar{l}}_{1}:1\leq \bar{l} \leq n\}$, the FTMC are
\begin{center}
 \begin{tabular}{|m{14.0em}|c|c|c|c|}
 \hline
 ${\ss}_{f}(j^{\bar{l}}_{1}|R_{f})$                            & $j^{1}_{1}$    & $j^{2}_{1}$    & $j^{\bar{g}+1}_{1}$   & $j^{\bar{g}+2}_{1}$\\
 \hline
 ${\ss}_{f}(j^{\bar{l}}_{1}|R_{f})$:($\bar{l}=1$)              & $\bar{l}-1$    & $1$            & $\bar{g}-\bar{l}+1$   & $\bar{g}-1$ \\
 \hline
 ${\ss}_{f}(j^{\bar{l}}_{1}|R_{f})$:($2\leq \bar{l}\leq \bar{g}+1$)  & $\bar{l}-1$ & $\bar{l}-2$    & $\bar{g}-\bar{l}+1$& $\bar{g}-\bar{l}+2$  \\
 \hline
 ${\ss}_{f}(j^{\bar{l}}_{1}|R_{f})$:($\bar{l}=\bar{g}+2$)      & $2\bar{g}-\bar{l}+1$ & $\bar{l}-2$ & $\bar{l}-\bar{g}-1$& $\bar{g}-\bar{l}+2$  \\
 \hline
 ${\ss}_{f}(j^{\bar{l}}_{1}|R_{f})$:($\bar{g}+3\leq \bar{l}\leq 2\bar{g}$) & $2\bar{g}-\bar{l}+1$ & $2\bar{g}-\bar{l}+2$ & $\bar{l}-\bar{g}-1$   & $\bar{l}-\bar{g}-2$  \\
 \hline
 \end{tabular}
 \end{center}
For the vertices $\{j^{\bar{l}}_{2}:1\leq \bar{l} \leq n\}$, the FTMC are
\begin{center}
 \begin{tabular}{|m{14.0em}|c|c|c|c|}
 \hline
 ${\ss}_{f}(j^{\bar{l}}_{2}|R_{f})$                         & $j^{1}_{1}$   & $j^{2}_{1}$  & $j^{\bar{g}+1}_{1}$   & $j^{\bar{g}+2}_{1}$\\
 \hline
 ${\ss}_{f}(j^{\bar{l}}_{2}|R_{f})$:($\bar{l}=1$)           & $\bar{l}$     & $1$      & $\bar{g}-\bar{l}+1$   & $\bar{g}$ \\
 \hline
 ${\ss}_{f}(j^{\bar{l}}_{2}|R_{f})$:($2\leq \bar{l}\leq \bar{g}$)  & $\bar{l}$ & $\bar{l}-1$  & $\bar{g}-\bar{l}+1$   & $\bar{g}-\bar{l}+2$  \\
 \hline
 ${\ss}_{f}(j^{\bar{l}}_{2}|R_{f})$:($\bar{l}=\bar{g}+1$)          & $2\bar{g}-\bar{l}+1$ & $\bar{l}-1$ & $\bar{l}-\bar{g}$ & $\bar{g}-\bar{l}+2$  \\
 \hline
 ${\ss}_{f}(j^{\bar{l}}_{2}|R_{f})$:($\bar{g}+2\leq \bar{l}\leq 2\bar{g}$) & $2\bar{g}-\bar{l}+1$ & $2\bar{g}-\bar{l}+2$ & $\bar{l}-\bar{g}$   & $\bar{l}-\bar{g}-1$  \\
 \hline
 \end{tabular}
 \end{center}
For the vertices $\{j^{\bar{l}}_{3}:1\leq \bar{l} \leq n\}$, the FTMC are
\begin{center}
 \begin{tabular}{|m{14.0em}|c|c|c|c|}
 \hline
  ${\ss}_{f}(j^{\bar{l}}_{3}|R_{f})$                             & $j^{1}_{1}$     & $j^{2}_{1}$    & $j^{\bar{g}+1}_{1}$   & $j^{\bar{g}+2}_{1}$\\
 \hline
 ${\ss}_{f}(j^{\bar{l}}_{3}|R_{f})$:($\bar{l}=1$)                & $\bar{l}+1$     & $2$            & $\bar{g}-\bar{l}+1$   & $\bar{g}-\bar{l}+2$ \\
 \hline
 ${\ss}_{f}(j^{\bar{l}}_{3}|R_{f})$:($2\leq \bar{l}\leq \bar{g}-1$) & $\bar{l}+1$  & $\bar{l}$      & $\bar{g}-\bar{l}+1$   & $\bar{g}-\bar{l}+2$  \\
 \hline
 ${\ss}_{f}(j^{\bar{l}}_{3}|R_{f})$:($\bar{l}=\bar{g}$)             & $\bar{l}+1$     & $\bar{l}$      & $2$             & $\bar{g}-\bar{l}+2$  \\
 \hline
 ${\ss}_{f}(j^{\bar{l}}_{3}|R_{f})$:($\bar{l}=\bar{g}+1$)           & $2\bar{g}-\bar{l}+1$  & $2\bar{g}-\bar{l}+2$ & $\bar{l}-\bar{g}+1$   & $2$  \\
 \hline
 ${\ss}_{f}(j^{\bar{l}}_{3}|R_{f})$:($\bar{g}+2\leq \bar{l}\leq 2\bar{g}-1$) & $2\bar{g}-\bar{l}+1$  & $2\bar{g}-\bar{l}+2$ & $\bar{l}-\bar{g}+1$   & $\bar{l}-\bar{g}$  \\
 \hline
 ${\ss}_{f}(j^{\bar{l}}_{3}|R_{f})$:($\bar{l}=2\bar{g}$)              & $2$   & $2\bar{g}-\bar{l}+2$ & $\bar{l}-\bar{g}+1$   & $\bar{l}-\bar{g}$  \\
 \hline
\end{tabular}
\end{center}
Finally, for the vertices $\{j^{\bar{l}}_{4}:1\leq \bar{l} \leq n\}$, the FTMC are ${\ss}_{f}(j^{\bar{l}}_{4}|R_{f})={\ss}_{f}(j^{\bar{l}}_{3}|R_{f})+(1,1,1,1)$, for $1\leq \bar{l} \leq n$. From these codes, we see that no two elements in $V(T_{n})$ have the same FTMC, implying $R_{f}$ to be a resolving set for $T_{n}$. Since, by definition of FTRS, the subsets $R_{f}\smallsetminus \{j\}$, $\forall$ $j\in R_{f}$ are $R_{1}=\{j^{1}_{1}, j^{2}_{1}, j^{\bar{g}+1}_{1}\}$, $R_{2}=\{j^{1}_{1}, j^{2}_{1}, j^{\bar{g}+2}_{1}\}$, $R_{3}=\{j^{1}_{1}, j^{\bar{g}+1}_{1}, j^{\bar{g}+2}_{1}\}$, and $R_{4}=\{j^{2}_{1}, j^{\bar{g}+1}_{1}, j^{\bar{g}+2}_{1}\}$. To prove that the set $R_{f}$ is the FTRS for $T_{n}$, we have to show that the sets $R_{1}$, $R_{2}$, $R_{3}$, and $R_{4}$ are the resolving sets for $T_{n}$. Then from the FTMC, as shown above, we find that the sets $R_{1}$, $R_{2}$, $R_{3}$, and $R_{4}$ are also resolving sets for $T_{n}$, as the metric coordinates for every distinct pair of vertices of $T_{n}$ are different with respect to the sets $R_{1}$, $R_{2}$, $R_{3}$, and $R_{4}$. This implies $fdim(T_{n})\leq 4$. Thus, from these lines, equation (1), and proposition 2 we have $fdim(T_{n})= 4$, in this case.\\
\textbf{Case(\rom{2})} $n\equiv 1(mod\ 2)$\\
For this, we set $n=2\bar{g}+1$, $\bar{g} \in \mathbb{N}$ and $\bar{g}\geq3$. Suppose $R_{f}=\{j^{1}_{1}, j^{2}_{1}, j^{\bar{g}+1}_{1}, j^{\bar{g}+3}_{1}\}\subseteq V(T_{n})$. Next, we give FTMC to every vertex of $T_{n}$ with respect to $R_{f}$. \\
For the vertices $\{j^{\bar{l}}_{1}:1\leq \bar{l} \leq n\}$, the FTMC are
\begin{center}
 \begin{tabular}{|m{14.0em}|c|c|c|c|}
 \hline
 ${\ss}_{f}(j^{\bar{l}}_{1}|R_{f})$                              & $j^{1}_{1}$    & $j^{2}_{1}$    & $j^{\bar{g}+1}_{1}$   & $j^{\bar{g}+3}_{1}$\\
 \hline
 ${\ss}_{f}(j^{\bar{l}}_{1}|R_{f})$:($\bar{l}=1$)                & $\bar{l}-1$    & $1$            & $\bar{g}-\bar{l}+1$   & $\bar{g}$ \\
 \hline
 ${\ss}_{f}(j^{\bar{l}}_{1}|R_{f})$:($2\leq \bar{l}\leq \bar{g}+1$) & $\bar{l}-1$ & $\bar{l}-2$    & $\bar{g}-\bar{l}+1$   & $\bar{g}-\bar{l}+2$  \\
 \hline
 ${\ss}_{f}(j^{\bar{l}}_{1}|R_{f})$:($\bar{l}=\bar{g}+2$)        & $2\bar{g}-\bar{l}+2$ & $\bar{l}-2$ & $\bar{l}-\bar{g}-1$& $\bar{g}-\bar{l}+2$  \\
 \hline
 ${\ss}_{f}(j^{\bar{l}}_{1}|R_{f})$:($\bar{g}+3\leq \bar{l}\leq 2\bar{g}+1$) & $2\bar{g}-\bar{l}+2$ & $2\bar{g}-\bar{l}+3$ & $\bar{l}-\bar{g}-1$   & $\bar{l}-\bar{g}-2$  \\
 \hline
 \end{tabular}
 \end{center}
For the vertices $\{j^{\bar{l}}_{2}:1\leq \bar{l} \leq n\}$, the FTMC are
\begin{center}
 \begin{tabular}{|m{14.0em}|c|c|c|c|}
 \hline
 ${\ss}_{f}(j^{\bar{l}}_{2}|R_{f})$                              & $j^{1}_{1}$    & $j^{2}_{1}$    & $j^{\bar{g}+1}_{1}$ & $j^{\bar{g}+3}_{1}$\\
 \hline
 ${\ss}_{f}(j^{\bar{l}}_{2}|R_{f})$:($\bar{l}=1$)                & $\bar{l}$      & $1$            & $\bar{g}-\bar{l}+1$ & $\bar{g}$ \\
 \hline
 ${\ss}_{f}(j^{\bar{l}}_{2}|R_{f})$:($2\leq \bar{l}\leq \bar{g}$)& $\bar{l}$      & $\bar{l}-1$    & $\bar{g}-\bar{l}+1$ & $\bar{g}-\bar{l}+3$  \\
 \hline
 ${\ss}_{f}(j^{\bar{l}}_{2}|R_{f})$:($\bar{l}=\bar{g}+1$)        & $2\bar{g}-\bar{l}+2$ & $\bar{l}-1$& $\bar{l}-\bar{g}$   & $\bar{g}-\bar{l}+3$  \\
 \hline
 ${\ss}_{f}(j^{\bar{l}}_{2}|R_{f})$:($\bar{l}=\bar{g}+2$)        & $2\bar{g}-\bar{l}+2$ & $2\bar{g}-\bar{l}+3$ & $\bar{l}-\bar{g}$   & $\bar{g}-\bar{l}+3$  \\
 \hline
 ${\ss}_{f}(j^{\bar{l}}_{2}|R_{f})$:($\bar{g}+3\leq \bar{l}\leq 2\bar{g}+1$) & $2\bar{g}-\bar{l}+2$ & $2\bar{g}-\bar{l}+3$ & $\bar{l}-\bar{g}$   & $\bar{l}-\bar{g}-2$  \\
 \hline
 \end{tabular}
 \end{center}
For the vertices $\{j^{\bar{l}}_{3}:1\leq \bar{l} \leq n\}$, the FTMC are
\begin{center}
 \begin{tabular}{|m{14.0em}|c|c|c|c|}
 \hline
 ${\ss}_{f}(j^{\bar{l}}_{3}|R_{f})$                                   & $j^{1}_{1}$ & $j^{2}_{1}$ & $j^{\bar{g}+1}_{1}$   & $j^{\bar{g}+3}_{1}$\\
 \hline
 ${\ss}_{f}(j^{\bar{l}}_{3}|R_{f})$:($\bar{l}=1$)                     & $\bar{l}+1$ & $2$         & $\bar{g}-\bar{l}+1$   & $\bar{g}-\bar{l}+2$ \\
 \hline
 ${\ss}_{f}(j^{\bar{l}}_{3}|R_{f})$:($2\leq \bar{l}\leq \bar{g}-1$)   & $\bar{l}+1$ & $\bar{l}$   & $\bar{g}-\bar{l}+1$   & $\bar{g}-\bar{l}+2$  \\
 \hline
 ${\ss}_{f}(j^{\bar{l}}_{3}|R_{f})$:($\bar{l}=\bar{g}$)               & $\bar{l}+1$ & $\bar{l}$   & $2$                   & $\bar{g}-\bar{l}+2$  \\
 \hline
 ${\ss}_{f}(j^{\bar{l}}_{3}|R_{f})$:($\bar{l}=\bar{g}+1$)             & $2\bar{g}-\bar{l}+2$ & $\bar{l}$  & $\bar{l}-\bar{g}+1$   & $2$  \\
 \hline
 ${\ss}_{f}(j^{\bar{l}}_{3}|R_{f})$:($\bar{g}+2\leq \bar{l}\leq 2\bar{g}$) & $2\bar{g}-\bar{l}+2$  & $2\bar{g}-\bar{l}+3$ & $\bar{l}-\bar{g}+1$   & $\bar{l}-\bar{g}$  \\
 \hline
 ${\ss}_{f}(j^{\bar{l}}_{3}|R_{f})$:($\bar{l}=2\bar{g}+1$)            & $2$  & $2\bar{g}-\bar{l}+3$ & $\bar{g}+1$         & $\bar{l}-\bar{g}$  \\
 \hline
\end{tabular}
\end{center}
Finally, for the vertices $\{j^{\bar{l}}_{4}:1\leq \bar{l} \leq n\}$, the FTMC are ${\ss}_{f}(j^{\bar{l}}_{4}|R_{f})={\ss}_{f}(j^{\bar{l}}_{3}|R_{f})+(1,1,1,1)$, for $1\leq \bar{l} \leq n$. From these codes, we see that no two elements in $V(T_{n})$ have the same FTMC, implying $R_{f}$ to be a resolving set for $T_{n}$. Since, by definition of FTRS, the subsets $R_{f}\smallsetminus \{j\}$, $\forall$ $j\in R_{f}$ are $R_{1}=\{j^{1}_{1}, j^{2}_{1}, j^{\bar{g}+1}_{1}\}$, $R_{2}=\{j^{1}_{1}, j^{2}_{1}, j^{\bar{g}+2}_{1}\}$, $R_{3}=\{j^{1}_{1}, j^{\bar{g}+1}_{1}, j^{\bar{g}+2}_{1}\}$, and $R_{4}=\{j^{2}_{1}, j^{\bar{g}+1}_{1}, j^{\bar{g}+2}_{1}\}$. To prove that the set $R_{f}$ is the FTRS for $T_{n}$, we have to show that the sets $R_{1}$, $R_{2}$, $R_{3}$, and $R_{4}$ are the resolving sets for $T_{n}$. Then from the FTMC, as shown above, we find that the sets $R_{1}$, $R_{2}$, $R_{3}$, and $R_{4}$ are also resolving sets for $T_{n}$, as the metric coordinates for every distinct pair of vertices of $T_{n}$ are different with respect to the sets $R_{1}$, $R_{2}$, $R_{3}$, and $R_{4}$. This implies $fdim(T_{n})\leq 4$. Thus, from these lines, equation (1), and proposition 2 we have $fdim(T_{n})= 4$, as well in this case.
\end{proof}

\begin{cor}
The FTMD for $T_{n}$ is constant.
\end{cor}

\hspace{-5.0mm}\textbf{{\Large $6$. Conclusion}}\\\\
We investigated the presence of independent fault-tolerant resolving sets in graphs and obtained some basic results comparing the independence of the resolving sets in this paper. We proved that $fdim(\mathbb{A}_{n})=fdim(S_{n})=fdim(T_{n})=4$, for the double antiprism $\mathbb{A}_{n}$, two convex polytopes $S_{n}$, and $T_{n}$. We end this section by posing a question as an open problem regarding IFTRS that derives naturally from the article. \\\\
\textbf{Open Problem:} {\it Is ifr$(\mathbb{A}_{n})=$ifr$(S_{n})=$ifr$(T_{n})=4$ ?}

\end{document}